\documentclass[10pt,twoside,english]{amsart}
\normalsize
\advance\oddsidemargin by -1.0cm
\advance\evensidemargin by -1.0cm
\usepackage{amssymb}
\usepackage{babel}
\usepackage{amsmath}
\usepackage{amscd}
\usepackage{epsfig}
\usepackage{psfrag}
\theoremstyle{plain}
\newtheorem{cor}{Corollary}[section]
\newtheorem{lem}{Lemma}[section]
\newtheorem{thm}{Theorem}[section]
\newtheorem{prop}{Proposition}[section]
\theoremstyle{definition}
\newtheorem{exa}{Example}[section]
\newtheorem{NB}{Remark}[section]

\newtheorem{dfn}{Definition}[section]

\newcommand{\bdm}{\begin{displaymath}}
\newcommand{\edm}{\end{displaymath}}
\newcommand{\be}{\begin{equation}}
\newcommand{\ee}{\end{equation}}
\newcommand{\ba}[1]{\begin{array}{#1}}
\newcommand{\ea}{\end{array}}
\newcommand{\bea}[1][]{\begin{eqnarray#1}}
\newcommand{\eea}[1][]{\end{eqnarray#1}}
\newcommand{\bqr}{\begin{eqnarray}}
\newcommand{\eqr}{\end{eqnarray}}
\newcommand{\bqrs}{\begin{eqnarray*}}
\newcommand{\eqrs}{\end{eqnarray*}}
\newcommand{\btab}{\begin{tabular}}
\newcommand{\etab}{\end{tabular}}

\newcommand{\x}{\times}

\newcommand{\tr}{\ensuremath{\mathrm{tr}}}

\newcommand{\univ}{\ensuremath{\mathrm{univ}}}
\newcommand{\tw}{\ensuremath{\mathrm{tw}}}

\newcommand{\cyclic}[1]{\stackrel{{\scriptsize #1}}{\mathfrak{S}}}
\newcommand{\C}{\ensuremath{\mathbb{C}}}

\newcommand{\T}{\ensuremath{\mathcal{T}}}
\newcommand{\TM}{\ensuremath{\mathcal{T}\!M}}
\newcommand{\Hol}{\ensuremath{\mathrm{Hol}}}
\newcommand{\hut}{\wedge}

\newcommand{\Ric}{\ensuremath{\mathrm{Ric}}}
\newcommand{\Scal}{\ensuremath{\mathrm{Scal}}}
\newcommand{\Scalg}{\ensuremath{\mathrm{Scal}^g}}

\newcommand{\U}{\ensuremath{\mathrm{U}}}

\newcommand{\SO}{\ensuremath{\mathrm{SO}}}
\newcommand{\Spin}{\ensuremath{\mathrm{Spin}}}

\newcommand{\Orth}{\ensuremath{\mathrm{O}}}

\newcommand{\Dslash}{\slash{\!\!\!\!D}}
\begin{document}
\def\haken{\mathbin{\hbox to 6pt{%
                 \vrule height0.4pt width5pt depth0pt
                 \kern-.4pt
                 \vrule height6pt width0.4pt depth0pt\hss}}}
    \let \hook\intprod
\setcounter{equation}{0}
%
%------ draw title page -----
%
\thispagestyle{empty}
%
%\hbox to \hsize{%
%  \vtop{} \hfill
%  \vtop{\hbox{PRELIMINARY VERSION}}}
%------------------------------
%\date{\today}
%---------------------------------------------------------------------------
\title{A note on generalized Dirac eigenvalues for split holonomy and torsion}
%---------------------------------------------------------------------------
%
% author and address
%
%-------------------------------------------
%
\author{Ilka Agricola}
\author{Hwajeong Kim}
\address{\hspace{-5mm}
Ilka Agricola \newline
Fachbereich Mathematik und Informatik \newline
Philipps-Universit\"at Marburg\newline
Hans-Meerwein-Strasse \newline
D-35032 Marburg, Germany\newline
{\normalfont\ttfamily agricola@mathematik.uni-marburg.de}}

\address{\hspace{-5mm}
Hwajeong Kim \newline
Department of Mathematics\newline
Hannam University\newline
Daejeon 306-791, Republic of Korea\newline
{\normalfont\ttfamily hwajkim@hnu.kr}\newline
}
\keywords{Dirac operator, eigenvalue estimate, metric connection with torsion}
\thanks{This work was supported by NRF(National Research Foundation of Korea) Grant funded by the Korean Government(NRF-2010-0002925)}
\subjclass{(MSC 2010):  53 C 25-29; 58 J 50; 58 J 60}
%-------------------------------------------
%
\pagestyle{headings}
\begin{abstract}
We study the Dirac spectrum on compact
Riemannian spin manifolds $M$ equipped with a metric connection
$\nabla$ with skew
torsion $T\in\Lambda^3M$ in the situation where
the tangent bundle splits under the holonomy of $\nabla$ and the
torsion of $\nabla$ is of `split' type. We prove
an optimal lower bound for the first eigenvalue of the
Dirac operator with torsion  that generalizes Friedrich's classical
Riemannian estimate.

\end{abstract}
%
%-------------
\maketitle
%----------------
%\tableofcontents
%----------------
%

%
%-------------- body of the document ------------------------------------------
%
%--------------------------------------------------------------------------
\section{Introduction }
%---------------------------------------------------------------------------
%
It is well-known that for a Riemannian manifold $(M,g)$, the fact that
 the holonomy representation of the Levi-Civita connection decomposes
in several irreducible modules has strong consequences for the geometry
of the manifold -- by de Rham's theorem, the manifold is locally a product,
and the spectrum of the Riemannian Dirac operator $D^g$  can be controlled
(\cite{ECKim04}, \cite{Alexandrov07}.

If the Levi-Civita connection is replaced by a metric connection with
torsion, not much is known, neither about the holonomy nor about the
implications for the spectrum. This note is a contribution to the much larger
task to improve our understanding of the holonomy of metric connections
with skew-symmetric torsion. The foundations of the topic were laid in
\cite{Agricola&F04a} (see also the review \cite{Agricola06}), substantial
progress on the holonomy in the irreducible case was achieved in
\cite{OR12} and \cite{Nagy13}. If the connection $\nabla$ is geometrically
defined, that is, it is the characteristic connection of some $G$-structure
on $(M,g)$, one is interested in the spectrum of the associated
characteristic Dirac operator $\Dslash$, a direct generalization of the
Dolbeault operator for a hermitian manifold and Kostant's cubic Dirac
operator for a naturally reductive homogeneous space. In \cite{Agricola&F&K08},
it was outlined that the first eigenvalue of  $\Dslash$ may be estimated from
below if the torsion is parallel; however, the paper could only deal with
$G$-structures on a case by case basis. A first general eigenvalue
estimate depending only on the connection $\nabla$ was given by means of
twistor theory in the authors' joint paper \cite{ABK11} with J.~Becker-Bender.
In this paper, we also examined the case that the manifold was reducible,
that the holonomy of the connection $\nabla$ decomposed accordingly, and
that the torsion $3$-form had no mixed parts.

This note is the first article devoted to the situation that the holonomy
representation of $\nabla$ on $TM$ splits into several irreducible submodules,
but the underlying manifold can yet be irreducible. We can then
prove an optimal eigenvalue estimate (Corollary \ref{estimate} )for the
Dirac operator $\Dslash$ under the assumption
that the torsion has no non-trivial contribution on any of the
$\nabla$-parallel distributions of $TM$ -- we call  $(M,g,\nabla)$ then a
manifold with \emph{split holonomy} (see Definition \ref{dfn.split-geom}).
Examples show that such geometries
arise quite naturally in the investigation of $G$-structures, a fact that
had not been observed before.
%
%----------------------------------------------------------------------------
\section{The estimate}
%
%----------------------------------------------------------------------------
%
\subsection{Geometric set-up}
%-----------------------------
We assume that $(M^n,g)$ is an oriented Riemannian
manifold endowed with a metric connection $\nabla$ with skew-symmetric torsion
$T\in\Lambda^3(M^n)$,
\bdm
\nabla_X Y\ :=\ \nabla^{g}_X Y +  \frac{1}{2} \cdot T(X,Y,-).
\edm
The holonomy group $\Hol(M^n;\nabla)$
is then a subgroup of $\SO(n)$, and we shall assume that it is a closed
subgroup to
avoid pathological cases. In order to distinguish it from the torsion,
the tangent bundle and its subbundles will be denoted by $\TM^n$,
$\T_1,\T_2\ldots$.
Recall that for a $\nabla$-parallel distribution, the standard proof of the
following basic lemma carries over from Riemannian geometry
 without modifications (see for example \cite[Prop.\,5.1]{Kobayashi&N1}).
\begin{lem}\label{basic-lemma}
%-----------------------------
Let $\T\subset\TM^n$ be a parallel distribution and $Y\in \T$. For any $X\in
\TM^n$, $\nabla_X Y$ is again in $\T$; in particular,
$R(X_1,X_2)Y\in\T$ for any $X_1,X_2$.
\end{lem}
Let $\T$ be a parallel distribution, $\mathcal{N}$ its orthogonal distribution defined
by $\mathcal{N}_x:=\T_x^\perp$ in every point $x\in M^n$. The fact that all elements
of $\Hol(M^n;\nabla)$ are orthogonal transformations implies that
$\mathcal{N}$ is again a parallel distribution. Thus, the tangent bundle
splits into an orthogonal sum of parallel distributions ($n_i:=\dim\T_i$)
\bdm
\TM^n\ =\ \T_1\oplus\ldots\oplus \T_k, \text{ and }
\Hol(M^n;\nabla)\subset \Orth(n_1)\x\ldots\x\Orth(n_k)\subset\SO(n).
\edm
We assume that every distribution $\T_i$ is again orientable and that
the holonomy preserves the orientation, i.\,e.~\emph{we assume}
\bdm
\Hol(M^n;\nabla)\subset \SO(n_1)\x\ldots\x\SO(n_k).
\edm
We denote an orthonormal frame of $\T_i$ by
$e^i_{1},\ldots,e^i_{n_i}$, $i=1,\ldots, k$.
For convenience, we assume that the spaces $\T_i$ are numbered by ascending
order, $n_1\leq n_2\leq \ldots\leq n_k$.
We recall the following properties of the curvature of the connection $\nabla$
from our previous article \cite{ABK11}:
\begin{enumerate}
\item
Since the distributions $\T_i,\T_j$ are orthogonal,
Lemma \ref{basic-lemma} implies for any vector fields $X,Y$ that
$g(\mathcal{R}(X,Y)\T_i,\T_j)=0 \text{ if } i\neq j$.
\item The Ambrose-Singer theorem implies that the curvature operator
$R(X,Y)$ vanishes if  $X\in\T_i,\ Y\in \T_j,\ i\neq j$.
\item The Ricci tensor has block structure,
\bdm
\Ric\ =\ \left[ \begin{array}{c|c|c} \Ric_1 & 0&\\ \hline 0 &\ddots& 0\\
    \hline  & 0 & \Ric_k
 \end{array}\right],
\edm
  i.\,e.~$\Ric(X,Y)\neq 0$ can only happen if  $X,Y\in\T_i$ for some $i$.
\item The scalar curvature splits into `partial scalar curvatures'
$\Scal_i:=\tr\, \Ric_i$, and ${\displaystyle\Scal=\sum_{i=1}^k \Scal_i}$.

\end{enumerate}
Be cautious that despite of the block structure of the Ricci curvature,
one has in general that  $R(X,Y,U,V)\neq 0$ if $X,Y\in \T_i,\ U,V\in  \T_j$
for $i\neq j$.
The space of $3$-forms splits under the holonomy representation into
\bdm
\Lambda^3(\T)\ =\ \bigoplus_{i=1}^k\Lambda^3(\T_i)\oplus
\bigoplus_{i\neq j}\Lambda^2(\T_i)\wedge \T_j\oplus \bigoplus_{i<j<k}
\T_i\wedge\T_j\wedge\T_k
\edm
In our first paper \cite{ABK11}, we treated in detail the situation that the
torsion $T$ of the connection $\nabla$ is entirely contained
in the first summand, i.\,e.~may be written as a sum $T=\sum_i T_i$
with $T_i\in\Lambda^3(\T_i)$. This is basically the case when $M$ is locally
a product.

The main point of this note is the observation that the other
extreme case, i.\,e.~that $T$ consists only of terms of the third type,
can also be controlled and is in fact not so exotic as it may
appear. Examples will be given in the last section. Thus, we define:
\begin{dfn}\label{dfn.split-geom}
%--------------------------------------
If the torsion $T$ satisfies $T(X,Y)=0$ whenever $X,Y\in\T_i$ and $\nabla T=0$,
we shall call $(M,g,\nabla)$ a manifold with \emph{split holonomy}.
\end{dfn}
Although the definition would make sense without the additional assumption
$\nabla T=0$, we shall see in the sequel that our method for estimating Dirac
eigenvalues relies strongly on this condition. Obviously, interesting
split geometries ($T\neq 0$) can only exist if $k\geq 3$, i.\,e.~the tangent bundle
splits into at least  three subbundles.
\begin{exa}
%----------
A metric almost contact manifold $M$ of dimension $2n+1$ has structure group
$\U(n)$, embedded as upper $(2n)\x (2n)$-matrices in $\Orth(2n+1)$. Thus,
the holonomy of a characteristic connection (if existent) is necessarily reducible,
the tangent bundle $\T M$ splits into a $2n$-dimensional and a one-dimensional
 parallel distribution. This is not yet sufficient for a manifold with
split holonomy; but in many cases, $\T M$ decomposes  further  with a torsion
of split type (see Section \ref{sec:examples}). On the other side, a
strict $G_2$-manifold or $\Spin(7)$-manifold (i.\,e.~without further
reduction to a subgroup $G\subset G_2,\ \Spin(7)$) cannot be of split
holonomy, since  $G_2$ and $\Spin(7)$ act by an irreducible representation.
\end{exa}
%
%-------------------------------------------------------------------------------
\subsection{Dirac operators and Schr\"odinger-Lichnerowicz formulas}
\noindent
%-------------------------------------------------------------------------------
%
Let us assume from now on that $M$ is also a  spin manifold.
Let $p_i$ denote the orthogonal projection from $\TM^n$ onto
$\T_i$ and define the `partial connections'
\bdm
\nabla^i_X\ :=\ \nabla_{p_i(X)} ,\quad \text{hence }
\nabla\ =\ \sum_{i=1}^k \nabla^i.
\edm
We use the same notation for their lifts to the spinor bundle $\Sigma M$.
They induce the notions of `partial Dirac operators' and `partial spinor Laplacians'
($\mu$ is the usual Clifford multiplication) through
\bdm
D_i\ :=\ \mu\circ\nabla^i, \quad D\ =\ \sum_{i=1}^k D_i,\quad
\Delta^i \ :=\ (\nabla^i)^*\nabla^i,\quad \Delta\ =\ \sum_{i=1}^k \Delta^i.
\edm
As long as the connection is not further specified, this is a correct
definition; if $\nabla$ is chosen to be an invariant connection for a
$G$ structure, i.\,e.~a characteristic connection, the `right'
Dirac operator to consider is the
characteristic Dirac operator $\Dslash$ associated with the connection with
torsion  $T/3$. Nevertheless, we shall also use $D_i$ and $D$ as an
intermediate tool.

At a fixed point $p\in M^n$ we choose orthonormal bases
$e^i_1,\ldots,e^i_{n_i}$  of the distributions $\T_i$ ($i=1,\ldots,k$) such that
$(\nabla_{e^i_m}e^j_l)_p = 0 $ for all suitable indices $i,j,m,l$.
This means in particular that
 $[e^i_m,e^j_l]=-T(e^i_m,e^j_l)$ and $\nabla^g_{e^i_m}e^i_m=0$. Denoting
$\nabla_{e^i_m}$ by $\nabla^i_m$, the partial Dirac and Laplace operators
may then be expressed as
\bdm
D_i\ :=\ \sum_{m=1}^{n_i} e^i_{m}\nabla^i_m,\quad
\Delta^i \ :=\ -\sum_{m=1}^{n_i} \nabla^i_m\nabla^i_m.
\edm
The divergence term of the Laplacian vanishes because of
$\nabla^g_{e^i_m}e^i_m=0$.
We compute the squares of the partial Dirac operators $D_i$.
\begin{prop}\label{partial-SL-1}
%-------------------------------
If $(M,g,\nabla)$ is a manifold with split holonomy,
the partial Dirac operators $D_i$ satisfy the identities
\bdm
(D_i)^2\psi \, =\, \Delta^i\psi + \tilde{\sigma}^i_T
+ \frac{1}{4}\tau_i\cdot\psi,
\edm
where
\bqr\label{4-form}
\tilde{\sigma}^i_T = \
=\ \frac{1}{2}\sum_{
%\left(
\begin{array}{ccc}
  k<l, p<q,&   \\
 e^i_ke^i_le_pe_q \ 4 \text{-form} &
\end{array}
%\right)
 } \!\!\!\!\!\! R(e^i_k,e^i_l,e_p,e_q)
e^i_ke^i_le_pe_q
\eqr
for any  numbering $\{e_p\}_{p=1,\cdots, n}$ of the total
 orthonormal frame $\cup_{i=1}^{k} \{e^i_1,\ldots,e^i_{n_i}\}$.
\end{prop}
\begin{proof}
%------------
For the first identity, let $k$ and $l$ be indices running
between $1$ and $\dim \T_i=n_i$. We  split the sum into
terms with $k=l$ and $k\neq l$,
\bdm
(D_i)^2\psi \, =\, \sum_{k,l=1}^{n_i} e^i_k\nabla^i_k e^i_l\nabla^i_l\psi
\ =\  -\sum_{k=1}^{n_i} \nabla^i_k\nabla^i_k\psi+\sum_{k\neq l}
e^i_ke^i_l\nabla^i_k\nabla^i_l\psi
\ =\ \Delta^i+\sum_{k< l}e^i_ke^i_l(\nabla^i_k\nabla^i_l-
\nabla^i_l\nabla^i_k)\psi
\edm
und express the second term through the curvature in the spinor bundle,
\bdm
(D_i)^2\psi \, =\, \Delta^i\psi + \sum_{k<l} e^i_ke^i_l\left[
R^\Sigma(e^i_k,e^i_l) - \nabla_{T(e^i_k,e^i_l)}\right] \psi.
\edm
By our assumption of split holonomy, $T(e^i_k,e^i_l)=0$, so the corresponding
term vanishes.
$R^\Sigma$ in turn can be expressed through the curvature $R$ (see
\cite{Agricola03}, \cite{ABK11}), and, by the curvature properties listed
before, only terms with all four vectors inside $\T_i$ can occur:
\bdm \sum_{k<l} e^i_ke^i_l R^\Sigma(e^i_k,e^i_l)\ =\
\frac{1}{2}\sum_{k<l} e^i_ke^i_l R(e^i_k\wedge e^i_l)\cdot \psi\
=\ \frac{1}{2}\sum_{k<l, p<q} R(e^i_k,e^i_l,e_p,e_q)
e^i_ke^i_le_pe_q\psi. \edm
Note here that $e_p,e_q$ are not necessarily from $\T_i$.
The summands with same indices add up to  half the partial scalar curvature,
while different indices  yield the Clifford multiplication by the $4$-form $\tilde{\sigma}_T$ by (\ref{4-form}),
\bdm
\sum_{k<l} e^i_ke^i_l R^\Sigma(e^i_k,e^i_l)\
=\tilde{\sigma}^i_T+ \frac{1}{4}\tau_i. \qedhere
\edm
\end{proof}
Recall that the characteristic Dirac operator $\Dslash^2$ is linked
to the Laplacian of the connection $\nabla$ through the following
Schr\"odinger-Lichnerowicz formula (\cite{Bismut}, \cite{Agricola&F04a}).
Here, $\Scalg$ and
$\Scal$ denote the scalar curvatures of the Levi-Civita connection
and the new connection $\nabla$, respectively, and
\bdm
\sigma_{T} \ :=\ \frac{1}{2} \sum_k (e_k\haken T)\hut (e_k\haken
T).
\edm
%
%---------------------------------------------------------
\begin{thm}\label{new-weitzenboeck}
For $\nabla T=0$, the spinor Laplacian $\Delta$ and the square
of the Dirac operator $\Dslash$ are related by
\begin{equation}\label{D3-2} \Dslash^2
 \ = \ \Delta -\frac{1}{4}\,T^2 +\frac{1}{4}\,\Scalg + \frac{1}{8} \|
          T\|^2
 \ = \ \Delta^c + \sigma_{T} + \frac{1}{4}\Scal+ \frac{1}{4}T^2
\end{equation}
\end{thm}
%\begin{proof}
%From Theorem (?).
%\end{proof}
%
Then the Dirac operators $\Dslash$ and $D^c_i$ satisfy the
following relationship:
\begin{prop}\label{partial-SL-2}
%-------------------------------
If $(M,g,\nabla)$ is a manifold with split holonomy,  we have
\be\label{sigma}
\sum_{i=1}^k\tilde{\sigma}^i_T = \sigma_T,
\ee
which implies
\begin{equation}\label{D3-3}
 \sum_{i=1}^k (D_i)^2\psi \ = \
\Delta \psi +  \sigma_T\psi + \frac{1}{4}\Scal\,\psi
\ = \ \Dslash^2 \psi- \frac{1}{4}T^2\psi.
\end{equation}
\end{prop}
\begin{proof}
%------------
For the identity (\ref{sigma}), observe that $\nabla T=0$
implies $dT=2\sigma_T$, hence the first Bianchi identity is
reduced to
\[\cyclic{X,Y,Z} R(X,Y,Z,V)\ =
\sigma_T(X,Y,Z,V).\]
>From the symmetry property of $R(X,Z,U,V)$ with respect to $X,Y$
and $U,V$ and  Lemma \ref{basic-lemma}, it holds that $R(e^i_m,
e^j_l)=0$, for $i\neq j$. Thus, we have the following equation for
the $4$-form and the partial $4$-forms:
\bqrs
 \sigma_T & = & \frac{1}{2}\sum_{p<q,r<s}R(e_p,e_q,e_r,e_s)e_pe_qe_re_s\\
          & = & \sum_i\frac{1}{2}\sum_{p<q,r<s, e_p,e_q\in\T_i}
          R(e_p,e_q,e_r,e_s)e_pe_qe_re_s=\sum_i\tilde{\sigma}^i_T.
\eqrs
 The equality (\ref{D3-3}) is then a consequence of
Proposition \ref{partial-SL-1}, (\ref{D3-2}) and (\ref{sigma}).
\end{proof}
%
%
%------------------------------------------------------------------------------------
\subsection{An Adapted Twistor Operator}\label{sec:twis}\noindent
%------------------------------------------------------------------------------------
%
For our eigenvalue estimate, the crucial point is to use an adapted twistor
operator. Define an operator $ P:\Gamma(\Sigma
M) \longrightarrow  \Gamma (T^{\ast}\otimes\Sigma M)$ by
\bdm
P\psi \ :=  \ \nabla^c \psi +
\sum_{i=1}^{k}\frac{1}{n_i}\sum_{l=1}^{n_i}e^i_l \otimes
e^i_l\cdot D^c_i\psi.
\edm
By  a direct computation, one checks
\bqr \label{twistor}
\|P\psi\|^2  = \int \langle(\Delta - \sum_{i=1}^k
                   \frac{1}{n_i}(D_i)^2)\psi, \psi\rangle dM.\label{twistor}
          % =  \big(\big(\frac{n_k-1}{n_k}(D^\frac{1}{3})^2 -
          %\sum_{i=1}^{k-1}(\frac{1}{n_i}-\frac{1}{n_k})(D^c_i)^2
          %  + \frac{1}{n_k}\sum_{i=1}^{k}\sum_{m<l}e^i_me^i_l\nabla^c_{T(e^i_m,e^i_l)} \psi
          %+ (\frac{1}{4n_k} + \frac{1}{4})T^2 -
          %\frac{1}{2}||T^2||\big)\psi,\psi\big),
\eqr
%
%Now, from (\ref{D3-2}), (\ref{D3-3}) and (\ref{twistor})
%
The crucial step is the following integral identity. Recall
that the dimensions $n_i$ of the distributions $\T_i$ are chosen to be
ordered, $n_1\leq n_2\leq\ldots\leq n_k$:
%
%------------------------------------
\begin{thm}\label{D/3-square-2}
Let $(M,g,\nabla)$ be a manifold of split holonomy.
Then the Dirac operator $\Dslash$ satisfies
\bea[*]
\int\left(\Dslash^2\psi,\psi\right)dM
   &=&\frac{n_k}{4(n_k - 1)}\int\left(\Scalg\psi,\psi\right)dM +
            \int\left(\left(\frac{n_k}{8(n_k - 1)}||T||^2 -
             \frac{1 + n_k}{4n_k - 4}T^2\right)\psi, \psi\right)dM\\
   & &  + \frac{n_k}{n_k - 1}||P\psi||^2 +
      \frac{n_k}{n_k - 1}\sum_{i=1}^{k-1}\left(\frac{1}{n_i} -
      \frac{1}{n_k}\right)\|(D^c_i)^2\psi\|^2 .
\eqrs
\end{thm}
\begin{proof}
%------------
>From the generalized Schr\"odinger-Lichnerowicz formula
\[\Dslash^2 = \Delta -\frac{1}{4}\,T^2 +\frac{1}{4}\,\Scalg + \frac{1}{8}
\|T\|^2.\,\]
So, we compute
\bdm
\Delta -\sum_{i=1}^{k}\frac{1}{n_i}(D_i)^2 \ = \
\Dslash^2
- \frac{1}{n_k}(D_k)^2 - \sum_{i=1}^{k-1}\frac{1}{n_i}(D_i)^2
- \left[ -\frac{1}{4}\,T^2 +\frac{1}{4}\,\Scalg + \frac{1}{8} \|T\|^2\right]
\edm
By equation (\ref{D3-3}), this can be rewritten
\bea[*]
\Delta -\sum_{i=1}^{k}\frac{1}{n_i}(D_i)^2 & = &
\Dslash^2 - \frac{1}{n_k}\Dslash^2 -
           \sum_{i=1}^{k-1}\left[\frac{1}{n_i}-\frac{1}{n_k}\right](D_i)^2
+ \frac{1}{4n_k}T^2  - \left[ -\frac{1}{4}\,T^2 +\frac{1}{4}\,\Scalg +
  \frac{1}{8} \|T\|^2\right] \\
& = & \left[\frac{n_k-1}{n_k}\right]\Dslash^2
-\sum_{i=1}^{k-1}\left[\frac{1}{n_i} -\frac{1}{n_k}\right](D_i)^2
- \frac{1}{4}\,\Scalg  + \left[\frac{1}{4n_k}+\frac{1}{4}\right]T^2
     -\frac{1}{8}\|T\|^2
\eea[*]
The identity (\ref{twistor}) for the the adapted twistor operator $P$ thus
implies the desired identity.
\end{proof}
We now recall the general Schr\"odinger-Lichnerowicz formula from
Theorem \ref{new-weitzenboeck}, which relates
$\Dslash^2$ and $\Delta^c$. Since the torsion $T$ is $\nabla$-parallel,
 $\Delta$ commutes with $T$, and we obtain (\cite{Agricola&F04b}, Proposition 3.4)
\bdm
\Dslash^2\circ T \ =\ T\circ \Dslash^2.
\edm
It is therefore possible to split the spin bundle $\Sigma M$ in the orthogonal sum
of its eigenbundles for the $T$ action,
\bdm
\Sigma M\ =\ \bigoplus_{\mu}\Sigma_{\mu},
\edm
and to consider $\Dslash^2$ on each of them, since
$\nabla^s$ and $\Dslash^2$ both preserve this splitting.
We shall henceforth denote the different eigenvalues of $T$ on $\Sigma$ by
$\mu_1,\ldots,\mu_l$. This method of evaluating eigenvalues was
first described in \cite{Agricola&F&K08}, see also \cite{Kassuba10}.
\begin{cor}\label{estimate}
%---------------------------
Let $\lambda$  be
an eigenvalue of $\Dslash^2$ with an eigenspinor $\psi$
which lies in $\mu$-eigenspace of $T$. Then,
\[\lambda(\Dslash^2|_{\Sigma_{\mu}}) \ge \frac{n_k}{4(n_k - 1)}\Scalg_{\text{min}} +
\frac{n_k}{8(n_k - 1)}||T||^2 - \frac{1 + n_k}{4(n_k - 1)}\mu^2:=\beta_{\ensuremath{\mathrm{split}}}(\mu).
\]
The equality holds if and only if $\Scalg$ is constant, $P(\psi) = 0$
and either $n^i = n^k$   or $D_i\psi = 0$, for all $i=1,\ldots, k$.
For the smallest eigenvalue $\lambda$ of $\Dslash^2$ on the whole
spin bundle $\Sigma M $, one thus obtains the estimate
\bdm
\lambda\ \geq\ \frac{n_k}{4(n_k - 1)}\Scalg_{\text{min}} +
\frac{n_k}{8(n_k - 1)}||T||^2 - \frac{1 + n_k}{4(n_k - 1)}\max(\mu_1^2,\ldots,\mu_k^2):=\beta_{\ensuremath{\mathrm{split}}}.
\edm
\end{cor}
\begin{proof}
%------------
The inequality is a direct consequence of  Theorem \ref{D/3-square-2}.
\end{proof}
\begin{NB}\label{other-estimates}
%---------------------------------
The eigenvalue estimate from \cite{ABK11} for reducible holonomy may not be
applied in this situation. However, two other general eigenvalues may be
compared to our result. Both require only the condition $\nabla T=0$, no
assumption on the holonomy:
\begin{enumerate}

\item In \cite{Agricola&F04a}, it is proved that
\bdm
\lambda\ \geq\ \frac{1}{4}\Scal^g_{\min} + \frac{1}{8}\|T\|^2 -\frac{1}{4}\,
\max(\mu_1^2,\ldots,\mu_k^2)\ =:\ \beta_{\univ}.
\edm
This is called the \emph{universal eigenvalue estimate}, because it
is derived from the universal Schr\"odinger-Lichnerowicz formula cited in
Theorem \ref{new-weitzenboeck}.
\item In the first part of \cite{ABK11}, twistor theory is used to
prove ($n:=\dim M$)
\bdm
\lambda\ \geq\ \frac{n}{4(n-1)}\Scal^g_{\min} + \frac{n(n-5)}{8(n-3)^2}\|T\|^2
+ \frac{n(4-n)}{4(n-3)^2}\,
\max(\mu_1^2,\ldots,\mu_k^2)\ =:\ \beta_{\tw}.
\edm
This estimate has the advantage that it yields the classical
Riemannian estimate by Friedrich from \cite{Friedrich80} if $T=0$.
\end{enumerate}
\end{NB}
\begin{NB}
%---------
It is interesting to ask what the `extreme' case would be for
our new eigenvalue estimate (Corollary \ref{estimate}). If there is only
parallel distribution, $\T=\T_1$ (i.\,e., $k=1$ and $n_1=\dim M$),
the condition of split holonomy requires $T=0$ (and in particular,
$\nabla T=0$ is trivially fulfilled). The estimate does then
coincide with Friedrich' estimate \cite{Friedrich80}, i.\,e.~it is the best
possible one.
\end{NB}
%
%------------------------------------------------------------------------------------
\section {Examples}\label{sec:examples}
%------------------------------------------------------------------------------------
%
Several examples will show that the assumption of split holonomy occurs quite
naturally in the study of $G$ structures on manifolds.
\begin{exa}
%----------
The twistor spaces of the only $4$-dimensional compact self-dual Einstein
manifolds $S^4$ and $\C\mathbb{P}^2$ are the $6$-dimensional manifolds
$\C\mathbb{P}^3$ and $F(1,2)=\U(3)/U(1)\times U(1)\times U(1)$, the manifold of
flags $l\subset v$ in $\C^3$ such that $\dim l=1$ and $\dim v=2$.
It is well-known that they carry two Einstein metrics; one is K\"ahler
(on $\C\mathbb{P}^3$, this is exactly the Fubini-Study metric), the other is
nearly K\"ahler. We shall henceforth be interested in their
nearly K\"ahler structure.  The characteristic connection $\nabla$ for nearly
K\"ahler manifolds was first considered by Gray in \cite{Gray70} and,
in this particular case, happens to coincide with the Chern connection
(see the review \cite{Gauduchon97} for general hermitian connections and
\cite{Friedrich&I2} for the general description of characteristic connections on
almost hermitian manifolds).
By a theorem of Kirichenko (\cite{Kirichenko77}, \cite{Alexandrov&F&S04}),
the torsion $T$ of $\nabla$ is parallel, $\nabla T=0$, which is the first
of the conditions needed for split holonomy.
In \cite{Belgun&M01}, it was proved that the only complete, 6-dimensional,
non-K\"ahler nearly K\"ahler manifolds such that the characteristic
connection has reduced holonomy are exactly $\C\mathbb{P}^3$ and $F(1,2)$
(as Riemannian manifolds, both are of course irreducible).
For computational details on these very interesting spaces, we refer to
\cite[Section 5.4]{BFGK}. In fact, one checks that in both cases, the holonomy
of $\nabla$ splits the tangent space in three two-dimensional subbundles
$\T^2_i$ (the upper index indicates the dimension)
\bdm
\T M\ =\ \T^2_1\oplus \T^2_2\oplus \T^2_3, \quad M=\C\mathbb{P}^3
\text{ or } F(1,2).
\edm
The general identities for  nearly K\"ahler manifolds imply that
$\Scal^g=30$, $\|T\|^2=4$ and $T$ has the eigenvalues $\mu=0$ and
$\mu=\pm 2\|T\|$. Furthermore,  there exist two Riemannian Killing spinors
$\varphi^\pm$ that satisfy $\Dslash \varphi^\pm = \mp \|T\|\varphi^{\pm} $
\cite{Friedrich&G85}.
To fix the ideas, in the notations of \cite[Section 5.4 a)]{BFGK} for
$M= F(1,2)$: $\T_1=\langle e_1, e_2\rangle,\ \T_2=\langle e_3, e_4\rangle,\
\T_3=\langle e_5, e_6\rangle$, the almost complex structure  and
the torsion $T$ of the characteristic connection $\nabla$ are
\bdm
\Omega\ =\ e_{12}-e_{34}+e_{56},\quad
T\ =\ e_{245}+ e_{146}-e_{236}+ e_{135}.
\edm
Here and in the sequel, we abbreviate exterior products $e_i\wedge e_j\wedge
\ldots$ as $e_{ij\ldots}$.
Thus, we are indeed in the situation of split holonomy as defined in
Definition \ref{dfn.split-geom}, and the eigenvalue estimate from
Corollary \ref{estimate} takes in this situation the value
\bdm
\lambda \ \geq\ \frac{2}{4(2-1)}\Scal^g + \frac{2}{8(2-1)}\|T\|^2 -
\frac{1+2}{4(2-1)}\max (0, 4 \|T\|^2)\ =\ 4\ =:\ \beta_{\text{split}}
\edm
Thus, one sees that our estimate is optimal in this situation, since
the two Killing spinors realize this lower bound. However,
the result could also have been obtained directly from \cite{Agricola&F04a},
since the bound $\beta_{\text{split}}$ coincides with the universal
eigenvalue estimate $\beta_\univ$ (see Remark \ref{other-estimates}). This is due
to the deeper fact that the two Killing spinors are in fact $\nabla$-parallel.
\end{exa}
\begin{exa}
%----------
In \cite{Schoemann}, the author classifies  $6$-dimensional almost hermitian
manifolds with parallel torsion by discussing the possible holonomy
groups of the characteristic connection (denoted by $\nabla^c$ in this paper)
and the normal form of the torsion. One finds that there are many more
examples of manifolds with split holonomy -- for example, all cases with
$\mathrm{Hol}(\nabla^c)\subset S^1, T^2$, of which there are many interesting
examples. However, it is not possible to test the eigenvalue estimate
from Corollary \ref{estimate} explicitly, since the curvature is not
fixed by these data.
\end{exa}
\begin{exa}
%----------
The Stiefel manifolds $M^5= \SO(4)/\SO(2)$ and $M^7=\SO(5)/\SO(3)$ carry
a normal homogeneous metric and a distinguished Sasaki structure; both are
described in detail in \cite{ABK11}, Example 5.1 (parameter value $t=1/2$ of
the metric)  and Example 5.2 (parameter value $t=1$ of the metric).
Both are well-known spaces in the investigation of Riemannian
spin manifolds: besides the metric that we are investigating, both
carry an Einstein-Sasaki metric and, therefore, they admit two Riemannian
Killing spinors  (\cite{Friedrich80} for $M^5$,
\cite{Kath00} for $M^7$).
The characteristic connection $\nabla$ turns out to be the canonical
connection of the underlying homogeneous space, hence the holonomy
representation coincides with the isotropy representation
(see \cite{Kobayashi&N2}) and the torsion is automatically
parallel (the space is naturally reductive).
The tangent bundle splits into (again, the upper index
denotes the dimension)
\bdm
\TM^5\ =\ \T^2_1\oplus \T^2_2\oplus \T^1_3,\quad
\TM^7\ =\ \T^3_1\oplus \T^3_2\oplus \T^1_3.
\edm
The Sasaki direction corresponds in both cases to the one-dimensional
bundle. With respect to a consecutive numbering of vectors of an orthonormal basis
(this coincides with the numbering from \cite{ABK11}), the torsion is
\bdm
T_{M^5}\ =\ - (e_{135}+e_{245}),\quad T_{M^7}\ =\ - (e_{147}+ e_{257}+e_{367}),
\edm
so one sees that again, the manifold is spin and of split holonomy.
There are two spinors that are constant under the lift of the isotropy
representation, thus they define global sections and they are
$\nabla$-parallel with Dirac eigenvalue $\lambda=1$. One easily checks
with the geometric data given in \cite{ABK11} that this is equal to
the bound given by all three known eigenvalue bounds,
\bdm
1\ =\ \beta_{\text{split}}\ =\ \beta_\univ\ =\ \beta_\tw.
\edm
This shows that our bound is, in this situation, again optimal.
We suspect that these examples can be generalized to the Tanno deformation
of any Einstein-Sasaki manifold: they have parallel torsion and
a natural splitting of the tangent bundle such that the torsion is of split
type, but it seems hard to prove in general that these subbundles are
indeed holonomy invariant. A description of the  Tanno deformation of an
Einstein-Sasaki manifold and of its characteristic connection may be found in
\cite{BB12}.
\end{exa}
%
%----------------------------------------------------------------------------
%------------------------------------------
%\addcontentsline{toc}{section}{Literature}
%------------------------------------------

\end{document}